\documentclass[12pt,a4paper]{article}
\usepackage[margin=1in]{geometry}
\usepackage{amsmath}
\usepackage{amsfonts}
\usepackage{amssymb}
\usepackage{amsthm}
\newtheorem{theorem}{Theorem}[section]
\newtheorem*{theorem*}{Theorem}

\newtheorem{lemma}[theorem]{Lemma}

\theoremstyle{definition}
\newtheorem{definition}{Definition}[section]
\newtheorem*{remark}{Remark}
\usepackage[colorinlistoftodos]{todonotes}
\usepackage{tikz}
\usepackage{tkz-fct}
\usetikzlibrary{shapes.geometric,positioning}
\usetikzlibrary{hobby}
\usepackage{pgfplots}
\usepackage{float}
\usepackage{enumerate}
\usepackage[utf8]{inputenc}
\usepackage[english]{babel}
\usepackage[normalem]{ulem}
\usepackage{setspace}
\usepackage {url}
\numberwithin{equation}{section}
\begin{document}

\title{The ergodic Mean Field Game system for a type of state constraint condition}
\author{Mariya Sardarli}
\date{}
\maketitle
\section{Introduction}
\noindent The object of this paper is to establish well-posedness (existence and uniqueness) results for a type of state constraint ergodic Mean Field Game (MFG) system  
\begin{align}\label{eq:proto}
\begin{cases}
-\Delta u+H(Du)+\rho =F(x;m) &\text{ in } \Omega,\\
\Delta m+div(m D_p H(Du))=0 &\text{ in } \Omega,\\
m\geq 0, \quad \int_\Omega m=1,
\end{cases}
\end{align}
supplemented with the state constraint-type infinite Dirichlet boundary condition 
\begin{equation}\label{intro:bds}
\lim\limits_{d(x)\to 0}u(x)=\infty,
\end{equation}
where $d(x):=d(x, \partial \Omega)$ is the distance to the boundary. Here $\Omega\subset \mathbb{R}^n$ is an open bounded subset of $\mathbb{R}^n$, $H: \mathbb{R}^n \times \overline{\Omega} \to \mathbb{R}$ is the Hamiltonian  associated with the cost function of an individual agent, $\rho$ is the ergodic constant, $m$ is the distribution of the agents, and $F: \Omega \times L^1(\Omega) \to \mathbb{R}$ is the interaction term. The Hamilton-Jacobi-Bellman (HJB) equation is understood in the viscosity sense and the Kolmogorov-Fokker-Planck (KFP) equation, in the distributional sense. 
\\\\
The MFG $(\ref{eq:proto})$ with the boundary condition $(\ref{intro:bds})$ can be interpreted as a state constraint-type problem. Indeed, the infinite boundary condition prevents the underlying stochastic trajectories from reaching the boundary, forcing them to stay in the domain. 
\\\\
The theory of MFGs was introduced by Lasry and Lions \cite{p1} and, in a particular setting, by Caines, Huang, and Malham\'e \cite{c1}, to describe the interactions of a large number of small and indistinguishable rational agents. In the absence of common noise, the behavior of the agents leads to a forward-backward coupled system of PDEs, consisting of a backward HJB equation describing the individual agent's value function, and a forward KFP equation for the distribution of the law (density) of the population. 
\\\\
The forward-backward system with periodic boundary conditions and either local or non-local coupling, as well its ergodic stationary counterpart, was first studied in \cite{p1}. Summaries of results may also be found in Cardaliguet \cite{notes}, Lions \cite{france}, and Gomes, Pimentel, and Voskanyan \cite{gom}. Fewer results exist, however, in the case of either Dirichlet, state constraint or Neumann boundary conditions.  
\\\\
We next state the main result of this paper for superlinear power-like Hamiltonians, that is, 
\begin{equation}
H(p,x)=|p|^q \text{ with } \quad q>1, \label{ham}
\end{equation}
and suitable assumptions on the coupling $F$ (see $(F1)$, $(F2)$ and $(F3)$ in Section 2.2). We emphasize that the particular form of the Hamiltonian is by no means essential to the analysis that follows. It simply provides a sufficiently general model problem that results in the necessary asymptotics of $u$ and $m$ near $\partial \Omega$. In particular, for $1<q\leq 2$, Hamiltonians $H(p)$ for which
\begin{equation}\label{eq:gen}
\delta^{q/(q-1)}H(\delta^{-1/(q-1)}p)- |p|^q \to 0,
\end{equation} 
locally uniformly in $p$ as $\delta \to 0$ are permissible. We state the result below, deferring the precise meaning of a solution to the system until Section 3.  
\begin{theorem*}\label{mainintro}
Let $1<q\leq 2$ and assume either $(F1)$ or $(F2)$. \\
Then, for all $r>1$, there exists a solution $(u,\rho, m)\in W^{2,r}_{loc}(\Omega) \times \mathbb{R} \times W^{1,r}_{loc}(\Omega)$ to the system  
\begin{align}\label{eq:mainintro}
\begin{cases}
-\Delta u+|Du|^q+\rho =F(x;m) &\text{ in } \Omega,\\
\lim\limits_{d(x)\to 0} u(x)=\infty, \\
\Delta m+div(m q|Du|^{q-2}Du)=0 &\text{ in } \Omega,\\
m \geq 0, \quad \int_\Omega m=1. 
\end{cases}
\end{align}
Moreover, if $F$ satisfies $(F3)$, then the solution is unique. 
\end{theorem*}
\noindent As will be explained below, the conditions $(F1)$ and $(F2)$ describe non-local and local couplings respectively, while $(F3)$ is the usual monotonicity condition.

\subsection{Background}
\noindent Heuristically, HJB equations with convex (in the momentum) Hamiltonian $H$ are associated with the stochastic control problem governed by the dynamics  
\begin{equation*}
dX_t=\alpha_t dt+\sqrt{2}dB_t,
\end{equation*}
where $(\alpha_t)_{t\geq0}$ is a non-anticipating process. In feedback form, the optimal policy is $-D_pH(Du)$. 
\\\\
In settings in which the trajectories of the agents reach the boundary but do not leave the domain or, as is the case we consider here, do not reach the boundary at all, the solutions of the HJB equation are referred to as state constraint solutions. 
\\\\
If the drift $-D_pH(Du)$ is bounded, then for all $t>0$, $P(X_t\in \Omega^c)>0$. It follows that for trajectories to remain inside $\Omega$, the drift must become unbounded near the boundary, pushing the agents back into the domain. For Hamiltonians of the form (\ref{ham}), or more generally $(\ref{eq:gen})$, state constraint solutions exist when the equation is paired with an infinite Neumann condition (reflection cost) or infinite Dirichlet condition (exit cost), the ``correct" choice of boundary condition depending on whether the Hamiltonian is sub- or super-quadratic. Indeed, when $q>2$, solutions are bounded but the normal component of the drift blows up at the boundary. In the sub-quadratic case, the solution itself also blows up. 
\\\\
In the sub-quadratic case, the trajectories never reach the boundary. Hence, the population density  is ``almost zero" near $\partial \Omega$. As the behavior of $m$ at the boundary is coupled with the blowup of $Du$ on $\partial \Omega$, no boundary condition is required on $m$ in (\ref{eq:mainintro}).  
\\\\
A type of MFG state-constraint problem has been studied by Porretta and Ricciardi in \cite{por2}. The authors impose structural assumptions on their Hamiltonian so that, for all choices of the control, an agent governed by these dynamics does not exit the domain $\Omega$. In this instance, no explicit boundary conditions on the value function are necessary to achieve well posedness. However, the assumptions of \cite{por2} do not permit coercive Hamiltonians, and, in particular, do not apply to the ones considered in this paper. 

\subsection{Infinite Dirichlet boundary conditions}
In the context of HJB equations, infinite boundary conditions with power-like Hamiltonians were studied by Lasry and Lions in $\cite{pll}$, who considered the HJB equation 
\begin{equation}
-\Delta u+|D u|^q+\lambda u =g \text{ in }\Omega,
\end{equation}
with $1<q\leq 2$ and $g\in L^\infty(\Omega)$, as well as the ergodic problem 
\begin{equation*}
-\Delta u+|D u|^q+\rho =g \text{ in }\Omega,
\end{equation*}
both coupled with the infinite Dirichlet boundary condition $(\ref{intro:bds})$. 
\\\\
It was shown in $\cite{pll}$ that the boundary value problem 
\begin{align}\label{eq:hjb}
\begin{cases}
-\Delta u+|D u|^q+\rho =g \text{ in }\Omega,\\
\lim\limits_{d(x) \to 0}u(x)=\infty,
\end{cases}
\end{align}
has a unique viscosity solution $(u,\rho)$, where uniqueness for $u$ is understood to be modulo an additive constant. Such a solution has unbounded drift in the direction of the boundary and the agents remain in the domain. \\\\
Due to the coupled nature of the MFG system, the existence and well posedness of the corresponding KFP equation depends in an essential way on the asymptotics of $u$ and $Du$ near the boundary. 
\\\\
In general, if $b\in L^1_{loc}(\Omega)$, the distributional solution $m$ of  
\begin{align}\label{eq:mbla}
\Delta m+div(mb)=0 \text{ in }\Omega, \quad \int_\Omega m=1, \quad m\geq 0,
\end{align}
need not be unique unless a boundary condition is specified. 
\\\\
On the other hand, if  
\begin{equation}\label{eq:13}
\displaystyle\lim\limits_{d(x)\to 0} b(x)d(x)=C \nu,
\end{equation}
where $C>0$ and $\nu$ is the outward unit normal to the boundary at the point closest to $x$, Huang, Ji, Liu and Yi showed in \cite{ua} that solutions are unique, despite the absence of boundary conditions. 
\\\\
To show $(\ref{eq:13})$ for $b=D_pH(Du)$ and apply this result, we need to make precise the asymptotics of $u$ and $Du$ near the boundary. 
\\\\
In \cite{pll}, the authors proved that if $(u, \rho)$ is a solution of $(\ref{eq:hjb})$, the asymptotics are precisely of this type. Namely, there exists a constant $C>0$ such that 
\begin{align*}
\lim\limits_{d(x)\to 0}\dfrac{D u(x)\cdot \nu (x)}{d^{-1/(q-1)}}=C.
\end{align*}
In this paper, we refine the result and show that the rate of convergence of the above limit is controlled by the $L^\infty-$norm of $g$. This becomes important when proving stability results related to $(\ref{eq:mainintro})$.

\subsection{The methodology} 
\noindent As usual, the existence of a solution to the MFG system $(\ref{eq:mainintro})$ is proven via a fixed point argument. To apply such an argument, it is crucial to obtain bounds on $m$ up to $\partial \Omega$.  Toward that end, for $\delta>0$, we define the subdomain $\Omega_\delta \subset \Omega$ by 
\begin{equation*}
\Omega_\delta:=\{x \in \Omega: d(x,\partial \Omega)>\delta\}.
\end{equation*}
As a first step, we show that if $b$ satisfies $(\ref{eq:13})$ with constant $C$, then there exists $\gamma=\gamma(C)>0$, $\delta=\delta(C)>0$, and $L=L(C,||b||_{L^\infty(\Omega_\delta)})>0$ such that any solution $m$ of $(\ref{eq:mbla})$ satisfies 
\begin{equation}\label{eq:gam|}
\int_\Omega d^{-\gamma} m\leq L.
\end{equation}
This estimate is needed in both the local and non-local cases.
\\\\
The space in which the fixed point argument is carried out differs in the non-local and local coupling settings. In the former, the fixed point is obtained in the space
\begin{equation}\label{wgamma}
\mathcal{W}_\gamma(\Omega):=\{m\in L^1(\Omega): \int_\Omega d^{-\gamma} m<\infty\},
\end{equation}
where $\gamma$ is as in $(\ref{eq:gam|})$ and, here, depends only on $q$. The map $T:\mathcal{W}_\gamma(\Omega)\to \mathcal{W}_\gamma(\Omega)$ formed by composing the solution maps of $(\ref{eq:hjb})$ and $(\ref{eq:mbla}$) satisfies the conditions of Schaefer's fixed point theorem. In particular, the continuity of $T$ is straightforward due to the regularizing effect of the coupling.
\\\\
By contrast, the continuity of the solution map is far from obvious in the local case, where $F(x;m)=f(m(x))$ for a continuous and bounded function $f$. Convergence of a sequence $m_n$ in $\mathcal{W}_\gamma(\Omega)$ does not yield local uniform convergence of $F(x;m_n)$ as it does in the non-local case. Thus the fixed point argument must be performed in a different space. 
\\\\
To deal with this issue, we introduce, for $\delta>0$, the approximating coupled ergodic system 
\begin{align}\label{eq:deltap}
\begin{cases}
-\Delta u_\delta+|Du_\delta|^q+\rho_\delta =f(\tilde{m}_\delta(x)),  &\text{ in } \Omega,\\
\lim\limits_{d(x)\to 0} u_\delta(x)=\infty,  \\
\Delta m_\delta+div\left(m_\delta q|D u_\delta|^{q-2}D u_\delta\right)=0  &\text{ in } \Omega_\delta,\\
\left(D m_\delta+m_\delta q|D u_\delta|^{q-2}D u_\delta\right)\cdot \nu=0 &\text{ on } \partial\Omega_\delta, \\
\int_\Omega m_\delta=1 ,
\end{cases}
\end{align}
where $(u_\delta, \rho_\delta, m_\delta) \in W^{1,r}\times \mathbb{R}\times C^{0,\alpha}(\Omega_\delta)$ and $\tilde{m}_\delta$ is a suitable extension of ${m_\delta}$ to all of $\Omega$ satisfying 
\begin{equation*}
||\tilde{m}_\delta||_{C^{0,\alpha}(\Omega)}=||m_\delta||_{C^{0,\alpha}(\Omega_\delta)}.
\end{equation*}
We note that (\ref{eq:deltap}) is not the standard MFG system with Neumann boundary conditions, in which the HJB and KFP equations are set in the same domain and both have Neumann boundary conditions. Instead, the HJB equation is set in the entire domain $\Omega$, allowing us to to take advantage of the known asymptotics of $u$ and $Du$ near the boundary. On the other hand, the Neumann conditions on the KFP equation allow us to exploit in regularity of $m_\delta$ up to $\partial \Omega_\delta$.\\\\
As we discuss below, for $b \in L^{\infty}(\Omega)$, the KFP equation 
\begin{align*}
\begin{cases}
\Delta m_\delta+div(m_\delta b)=0  &\text{ in } \Omega_\delta,\\
(D m_\delta+m_\delta b)\cdot \nu=0 &\text{ on } \partial\Omega_\delta,
\end{cases}
\end{align*}
has a a unique (modulo a multiplicative constant) distributional solution, which is positive and H\"older continuous up to the boundary with H\"older constant depending on the $L^\infty-$norm of the drift. If $u$ is the solution of a HJB equation like $(\ref{eq:hjb})$, $Du$ is bounded on $\Omega_\delta$ and this result applies to the KFP equation in $(\ref{eq:deltap})$. Moreover, as shown in \cite{pll}, the local bounds on the solutions $u_\delta$ of the HJB equation are uniform in the $L^\infty-$norm of $f$, which is independent of $m_\delta$. Therefore, as we show below, it is possible to carry out a fixed point argument in $C^{0, \alpha}(\Omega_\delta)$. 
\\\\
We are then able to pass from the solution $(u_\delta, \rho_\delta, m_\delta)$ of the approximating system to a solution $(u,\rho, m)$ of the original system on $\Omega$ by the stability of the HJB equation and the local uniform bounds on $m_\delta$, which in turn follow from local uniform bounds on $Du_\delta$.  
\\\\
Lastly, under the usual monotonicity assumption we establish uniqueness in both the local and non-local case. The key observation in the argument is that, for solutions $(\rho, u, m)$ of (\ref{eq:mainintro}), $m$ decays near the boundary as in (\ref{eq:gam|}).

\subsection{Organization of the paper}

\noindent The paper is organized as follows. The assumptions on the domain $\Omega$ and the coupling $F$ are stated in Section 2. Section 3 introduces the notions of weak solutions of the KFP and HJB equations and proves (local) regularity and stability results. Section 4 treats the non-local coupling case. In Section 5, we show the approximating problem (\ref{eq:deltap}) has a solution and pass to the limit to obtain a solution to $(\ref{eq:mainintro})$ on the entire domain. Section 6 establishes the uniqueness of the system and quantifies the sense in which $m$ vanishes at the boundary. The Appendix contains two technical lemmata used in Sections 3 and 5. 

\section{Assumptions and definitions}
\subsection{The domain}
Throughout this paper it is assumed that 
\begin{equation*}
\Omega \text{ is a bounded open subset of }\mathbb{R}^n\text{ with } C^2 \text{-boundary}.
\end{equation*} 
In particular, the domain satisfies the uniform interior ball condition. Namely, there exists a $\delta_0>0$ such that for every $x\in \partial \Omega$, the ball $B_{\delta_0}(x-\delta_0 \nu(x))$ is contained in $\Omega$. 
\\\\
Recall the definition of the subdomains 
\begin{equation*}
\Omega_\delta:=\{x\in \Omega: d(x,\partial\Omega)>\delta\}.
\end{equation*}
Due to the $C^2$-regularity of $\Omega$, there exists $\epsilon_0 \in (0,1)$ such that $d(x,\partial\Omega) \in C^2(\Omega\backslash \Omega_{\epsilon_0})$ and for all $x\in \Omega\backslash \Omega_{\epsilon_0}$ there exists a unique point $\overline{x}\in \partial \Omega$ such that $d(x, \partial\Omega)=|x-\overline{x}| $. Moreover, on $\Omega\backslash \Omega_{\epsilon_0}$,
\begin{equation*}
Dd(x, \partial \Omega)=Dd(\overline{x}, \partial \Omega)=-\nu (\overline{x}).
\end{equation*}
Lastly, we define $d$ to be a $C^2(\overline{\Omega})$ extension of $d(x,\partial\Omega)$ which is equal to $d(x,\partial\Omega)$ on $\Omega\backslash \Omega_{\epsilon_0}$. We additionally require that $d(x)\leq 1$ in $\Omega$ and $d(x)\geq M_0$ in $\Omega_{\epsilon_0}$ for some positive constant $M_0$.

\subsection{The coupling}
Our analysis permits a coupling $F:\Omega\times L^1(\Omega)\to\mathbb{R}$  which is either non-local, on the one hand, or local, continuous, and bounded on the other. 
Namely, we consider couplings $F$ satisfying one of the following:
\begin{enumerate} [(F1)]
\item The map $m \mapsto F(\cdot\ ; m)$ sends bounded sets in $L^1(\Omega)$ to bounded sets in $L^\infty(\Omega)$, and is continuous from $L^1(\Omega)$ into $L^\infty(\Omega)$.  
\item $F(x; m)=f(m(x))$ for $f$ a continuous, bounded function.
\end{enumerate} 
A coupling satisfying assumption $(F1)$ has an $L^\infty(\Omega)$-bound depending only on $\int_\Omega m$. A coupling satisfying $(F2)$ inherits regularity from the regularity of $m$ and is bounded by $||f||_{L^\infty(\mathbb{R})}$. \\\\
When proving uniqueness we will require the additional monotonicity assumption 
\begin{enumerate}[(F3)]
\item $\int_{\Omega}(F(x;m_1)-F(x;m_2))(m_1(x)-m_2(x))dx\geq 0$, with equality only if $m_1=m_2$ a.e.
\end{enumerate}

\subsection{The space of measures for the non-local setting}
In the introduction, we defined a space $\mathcal{W}_\gamma(\Omega)$ which quantifies the boundary behavior of solutions to KFP equations with an unbounded drift term. We repeat the definition here for completeness.
\begin{definition}
For $\gamma>0$, define a norm $||\cdot||_{\gamma,\Omega}$ by 
\begin{align*}
||m||_{\gamma,\Omega}=\int_\Omega d^{-\gamma}|m|,
\end{align*}
and recall that $\mathcal{W}_\gamma(\Omega)$ is the Banach space induced by this norm, as in $(\ref{wgamma})$.

\end{definition} 

\subsection{The extension of measures defined on subdomains}
In the local coupling setting discussed in Section 5 we will consider measures defined on the subdomain $\Omega_\delta$. Since $F$ is defined on $\Omega \times L^1(\Omega)$, it will be necessary to define a suitable extension of measures on $\Omega_\delta$ to $\Omega$ that preserves their H\"older regularity. 
\\\\
For $\mu\in C^{0,\alpha}(\Omega_\delta)$ and $x\in \Omega$ we define $\tilde{\mu}$ by  
\begin{align}\label{tilde}
\tilde{\mu}(x):=\inf_{y \in \Omega_\delta}\{\mu(y)+||\mu||_{C^{0,\alpha}(\Omega_\delta)}|x-y|^\alpha\}. 
\end{align}
It is straightforward to check that $\tilde{\mu}(x)=\mu(x)$ on $\Omega_\delta$ and that $\tilde{\mu}\in C^{0,\alpha}(\Omega)$ with the same H\"older modulus as $\mu$. 
\\\\
In the sections that follow, we abuse notation slightly and define the coupling $F: \Omega \times C^{0,\alpha}(\Omega_\delta) \to \mathbb{R}$ by 
\begin{align*}
F(x;\mu):=F(x;\tilde{\mu}).
\end{align*}

\section{Preliminaries}
In this section, we recall some existence and regularity results for the HJB equation and the KFP equation. At the end we prove stability results that will be used in Section \ref{s:fixed}.
\subsection{The Kolmogorov Fokker-Planck equation}\label{s:fp} Throughout this paper we will need to consider the KFP equation in subdomains of $\Omega$. To that end, we study  the following KFP boundary value problem for a general open bounded domain $V$ and $B\in L^\infty(V)$. 


\begin{eqnarray}
\begin{cases}\label{mdelta}
\Delta \mu +div(\mu B)=0 \ &\text{ in } V,\\
(D\mu +\mu B)\cdot \nu=0 &\text{ on }\partial V,\\
\int_{V} \mu=1,\, \mu>0 \, \text{ in }V.
\end{cases}
\end{eqnarray}
Later we will apply these results when $V=\Omega_\delta$ and $B=D_p(H(Du))$. 
\begin{definition}
Given $B\in L^{\infty}(V)$, we say that $\mu\in W^{1,2}(V)$ is a {Neumann weak solution} of $(\ref{mdelta})$ if $\mu$ is a positive probability measure such that, for all $\phi \in W^{1,2}(V)$, 
\begin{align*}
\int_{V}[D\phi(x)\cdot D\mu(x) +B(x)\cdot D \phi(x)\mu(x)]\ dx=0.
\end{align*}

\end{definition}
\noindent The next lemma gathers existence and uniqueness results for Neumann weak solutions of $(\ref{mdelta})$ as well as  $W_{loc}^{1,r}(V)$ estimates. 
\begin{lemma}\label{thm:mdelta}
For  $B\in L^{\infty}(V)$, $(\ref{mdelta})$ admits a unique Neumann weak solution $\mu\in W^{1,2}(V)$. In addition, $\mu\in W^{1,r}(V)$ for all $r>1$ and, for compact sets $K\subset K' \subset V$, $||\mu||_{W^{1,r}(K)}$ is bounded uniformly by a constant that depends only on $K, K'$ and $||B||_{L^{\infty}(K')}$. 
\end{lemma}
\begin{proof}
The existence, positivity, and uniqueness are proven in \cite{lady}. To obtain the local uniform $W^{1,r}$ estimate, we recall that Theorem 1.1 of \cite{rus1} yields a constant $C_1=C_1(K,  ||B||_{L^{\infty}(K')})>0$ such that  
\begin{equation} \label{eq:r}
||\mu||_{W^{1,r}(K)}\leq C_1||\mu||_{L^\infty(K)}.
\end{equation}
Then Harnack inequality and the normalization condition $\int_{V}\mu=1$ give some $C_2= C_2(K, K')>0$, such that
\begin{equation}\label{eq:r2}
\sup_{K} \mu \leq C_2 \inf_K \mu\leq  C_2 \frac{1}{|K|}. 
\end{equation}
Combining $(\ref{eq:r})$ and $(\ref{eq:r2})$ completes the proof. 
\end{proof}
\noindent If the drift in the KFP equation blows-up at the boundary, it is necessary to modify the notion of weak solution as follows. 
\begin{definition}\label{def:weak}
Given $b\in L^1_{loc}(\Omega)$ and $r>1$, we say  $\mu\in W^{1,r}_{loc}(\Omega)$ is a \emph{weak solution} of   
\begin{align}\label{eq:m}
\Delta \mu+div(\mu b)=0 \text{ in }\Omega,
\end{align}
if $\mu$ is a non-negative function such that, for all $\phi \in C^2_c(\Omega)$,
\begin{align*}
\int_{\Omega}[\Delta \phi(x)-b(x)\cdot D \phi(x)]\mu(x)\ dx=0.
\end{align*}
Moreover, $\mu$ is a \emph{proper weak solution} if $\int_\Omega \mu=1$. 
\end{definition}
\begin{remark}
The class of test functions in Definition \ref{def:weak} can be extended to $C^2$ functions that are constant in a neighborhood of $\partial\Omega$. \end{remark}

\noindent The following lemma shows that weak solutions of $(\ref{eq:m})$ may arise as the limit of Neumann weak solutions. The rest of the section provides conditions under which the limit is a proper weak solution.

\begin{lemma} \label{weakm}Let $m_\delta \in W^{1,r}(\Omega_\delta)$  be a Neumann weak solution of $(\ref{mdelta})$ for velocity fields $b_\delta\in L^\infty(\Omega_\delta)$. Assume that, as $\delta \to 0$, $b_\delta \to b$ locally uniformly. Then, along subsequences, $m_\delta \to m \in W^{1,r}_{loc}(\Omega)$ and $m$ is a weak solution of  
\begin{eqnarray*}
\Delta m+div(m b)=0 \text{ in } \Omega.
\end{eqnarray*}
\end{lemma}
 
\begin{proof}
Recalling that $||m_\delta||_{W^{1,r}(K)}$ is uniformly bounded for each compact $K\subset \Omega$, it is possible to extract a subsequence such that $m_\delta$ converges locally uniformly to a non-negative $W^{1,r}_{loc}(\Omega)$ function $m$. To see that $m$ satisfies the desired equation, consider a test function $\phi\in C^2_c(\Omega)$. Since $b_\delta$ converges uniformly to $b$ in the support of $\phi$, $\Delta \phi -b_\delta \cdot D \phi$ converges uniformly to $\Delta \phi -b \cdot D\phi$. As $m_\delta \to m$ uniformly in the support of $\phi$, letting $\delta \to 0$ in the weak formulation $\int_{\Omega_\delta}(\Delta \phi -b_\delta \cdot D \phi)m_\delta$ leads to the desired result. 
\end{proof}
\noindent It follows from Fatou's lemma that the limit $m$ in the above lemma is integrable on $\Omega$ and $\int_{\Omega}m\leq 1$. To ensure that $m$ is a proper weak solution it is necessary to impose additional structural conditions on the vector fields $b_\delta$. In \cite{rus1}, Bogachev and R{\"o}kner give conditions on the drifts under which the sequence of weak solutions $m_\delta$ is uniformly tight, which in turn yields $L^1(\Omega)$ convergence of $m_\delta$ to $m$. The version of the theorem relevant to our setting appears below. 
\begin{lemma}[Lemma 1.1 in \cite{rus1}]\label{thm:mpro}
Let $U_k$ by a sequence of nested domains and $\mu_k$ the corresponding sequence of Neumann weak solutions to $\Delta \mu_k+div(\mu_k b_k)=0$ on $U_k$. Suppose that $b_k\in L^1_{loc}(U_k)$ and that, for every $R>0$, there exists $\alpha_R>n$ such that 
\begin{align*}
\kappa_R:=\sup_k\int_{U_R}|b_k(x)|^{\alpha_R}dx<\infty.
\end{align*}
If, in addition, there exists $V\in C^2(U)$ such that, for some $c_k \to \infty$, $U_k=\{V<c_k\}$, $\{V=c_k\}$ has Lebesgue measure zero, 
\begin{align}\label{v1}
\lim_{d(x, \partial U)\to 0}V(x)=\infty, 
\end{align}
and
\begin{align}\label{v2}
\lim_{d(x, \partial U)\to 0}\sup_k [\Delta V- D V \cdot b_k]=-\infty,
\end{align}
then the sequence $\mu_k$ is uniformly tight. 
\end{lemma}
\noindent A function $V$ satisfying the above conditions is referred to as a Lyapanov function. The connection between Lyapanov functions and existence and uniqueness of weak solutions of $(\ref{eq:m})$ has been extensively studied in $\cite{rus1}$, as well as in $\cite{ua}$.  
\\\\
It follows from the Lemma \ref{thm:mpro} that, if there exists a $V$ satisfying $(\ref{v1})$, and 
\begin{align}\label{vb2}
\lim_{d(x, \partial U)\to 0}\Delta V- D V \cdot b=-\infty,
\end{align}
then, setting $b_k=b\chi_{U_k}$, we obtain a sequence of measures $m_k$ that converge to a proper weak solution $m$ of 
\begin{align}\label{eq:mbb}
\Delta m+div(m b)=0 \text{ in } \Omega.
\end{align}
It is proven in \cite{ua} that such a solution is unique.\\\\
In the present setting, these results will be applied by considering $b(x)$ satisfying, for some $C>1$ and $\epsilon \in (0,C-1)$,
\begin{align} \label{eq:asym}
\lim_{d(x)\to 0}(b(x)\cdot D d(x))d(x)=-C,
\end{align}
and $V(x)=d(x)^{-C+1+\epsilon}$, which satisfies $(\ref{v1})$ and $(\ref{vb2})$. This is stated precisely in the next result, the proof of which appears in the Appendix. 
\begin{lemma}\label{lemma:uni}
Given $b: \Omega\to \mathbb{R}^d$ satisfying $(\ref{eq:asym})$ with $C>1$, there exists a unique proper weak solution $m\in W^{1,r}_{loc}(\Omega)$ of   
\begin{align*}
\Delta m+div(m b)=0 \text{ in } \Omega,
\end{align*}
and, for every $\epsilon \in (0,C-1)$, there exists a $\delta =\delta(C, \epsilon)$, such that
\begin{align*}
\int_{\Omega} d^{-C-1+\epsilon}m\leq \hat{C}=\hat{C}(C, ||b||_{L^\infty(\Omega_\delta)}).
\end{align*}

\end{lemma}
\begin{remark}
The bound on $\int_{\Omega} d^{-C-1+\epsilon}m$ makes precise the intuition that $m$ ``approaches 0" near the boundary $\partial \Omega$. 
\end{remark}
\noindent It is proven in \cite{pll} that, for solutions $u$ of the HJB equation in $(\ref{eq:mainintro})$, $b=q|D u|^{q-2}D u$ satisfies precisely $(\ref{eq:asym})$ for $C=q/(q-1)$. This is discussed further in the next section. 
\subsection{The Hamilton-Jacobi-Bellman equation}
In this section, we fix $q\in (1,2]$ and $g\in L^\infty(\Omega)$ and consider the stationary HJB equation with infinite Dirichlet conditions 
\begin{eqnarray}\label{eq:hj}
\begin{cases}\label{udef}
-\Delta u+|D u|^q+\rho =g \text { in }\Omega,\\
\lim\limits_{d(x) \to 0} u(x)=\infty.  
\end{cases}
\end{eqnarray}
The estimates detailed here will be used to prove the HJB stability results in Theorem \ref{thm:hjstab}. 
\\\\
We begin by stating the definition of solutions for concreteness.

\begin{definition}
For $1<q\leq 2$ and $g\in L^\infty(\Omega)$, $(u, \rho) \in W^{2,r}_{loc}(\Omega)\times \mathbb{R}$ is an \emph{explosive solution} of $(\ref{udef})$ if $u$ is a viscosity solution in $\Omega$ and 
\begin{align*}
\lim_{d(x)\to 0^+}u(x)=\infty.
\end{align*}
\end{definition}
\noindent The existence and uniqueness of solutions to such HJB equations was proven in \cite{pll}, along with several local estimates. These results are collected in the next two theorems. The second theorem is a straightforward consequence of the bounds established in Theorem IV.I of \cite{pll}, with a slight modification to fit our ergodic framework. 
\begin{theorem}(Theorem VI.I in \cite{pll}))\label{thm: hj}
Let $1<q\leq 2$ and $g\in L^\infty(\Omega)$. The equation $(\ref{udef})$ admits an explosive solution $(u, \rho) \in W^{2,r}_{loc}(\Omega)\times \mathbb{R}$ for all $1<r<\infty$. 
\end{theorem}
  
\begin{theorem}\label{thm: bounds}
Let $1<q< 2$, $g\in L^\infty(\Omega)$ and $(u, \rho) \in W^{2,r}_{loc}(\Omega)\times \mathbb{R}$ be an explosive solution of $(\ref{eq:hj})$ with $u(x_0)=0$, for some fixed $x_0\in \Omega$ . There exist positive constants $C_1,C_2, C_3$, depending only on $\Omega$ and $||g||_{L^\infty(\Omega)}$, such that,
\begin{enumerate}[(i)]
\item $|\rho|\leq C_1$,
\item $|D u(x)|\leq C_2 d(x)^{-1/(q-1)}$, and 
\item $|u(x)|\leq C_3 d(x)^{(q-2)/(q-1)}$. 
\end{enumerate}
In addition, for any $K\subset K'$ compact subsets of $\Omega$, there exists a constant  $C_4=C_4(d(K,K'), n, r)>0$ such that 
\begin{enumerate}[(i)]
\item [(iv)]$||u||_{W^{2,r}(K)}\leq C_4(||u||_{L^\infty(K')}+||g||_{{L^q}(K')})$.
\end{enumerate} 
If $q=2$, the same estimates hold but $(iii)$ must be replaced by 
\begin{enumerate}[(i)]
\item [(iii)'] $|u(x)|\leq C_3 \log(d(x))$.
\end{enumerate}

\end{theorem}
\begin{proof}[Sketch of Proof of Theorem $\ref{thm: bounds}$]
The first result follows from the construction of $\rho$ in Theorem VI.I of \cite{pll} which shows that $\rho$ arises as the local uniform limit of $\lambda u_\lambda$ as $\lambda \to 0$, where $u_\lambda$ is the unique viscosity solution of 
\begin{align*}
\begin{cases}
-\Delta u_\lambda +|D u_\lambda|^q+\lambda u_\lambda&=g\text{ in }\Omega,\\
\lim\limits_{d(x)\to 0}u_\lambda(x)=\infty.
\end{cases}
\end{align*}
If $1<q<2$, for arbitrary $\epsilon>0$, $u_\lambda$ has as sub- and super-solutions $(C_q\mp\epsilon)d^{(q-2)/(q-1)}\mp C_\epsilon/ \lambda$, where $C_q=(q-1)^{(q-2)/(q-1)} (2-q)^{-1}$ and $C_\epsilon$ is a function of $||g||_{L^\infty(\Omega)}$. Similarly, for $q=2$, $u_\lambda$ has as sub- and super-solutions $-(1\mp\epsilon)\log(d)\mp C_\epsilon/ \lambda$, where $C_\epsilon$ is a function of $||g||_{L^\infty(\Omega)}$. Sending $\lambda\to 0$ yields the desired bound on $\rho$. \\\\ 
The second result $(ii)$  is proven in Theorem IV.I of \cite{pll} and $(iii)$ is an immediate consequence of $(ii)$ as $u(x_0)$ is fixed. The last interior estimate follows from a bootstrap argument. 
\end{proof}
\noindent In order to study the solution of the KFP equation near the boundary, it will be necessary to know the precise asymptotics of $D u$. In \cite{pll}, Lasry and Lions, and later Porretta and V{\'e}ron in \cite{por}, established the following results. 
\begin{theorem}\label{thm:asy}(Theorem I.1 and II.3 in \cite{pll}; Theorem 2.3 in \cite{por})
Let $1<q< 2$, $g\in L^\infty(\Omega)$ and $(u, \rho) \in W^{2,r}_{loc}(\Omega)\times \mathbb{R}$ be a solution of $(\ref{eq:hj})$. Then,
\begin{enumerate}[(i)]
\item $\lim\limits_{d(x)\to 0}\dfrac{u(x)}{d^{(q-2)/(q-1)}}=(q-1)^{(q-2)/(q-1)} (2-q)^{-1}$,
\item[and]
\item $\lim\limits_{d(x)\to 0}\dfrac{D u(x)\cdot \nu (x)}{d^{-1/(q-1)}}=(q-1)^{-1/(q-1)}.$
\end{enumerate}
\noindent The second result also holds for $q=2$, while the former must be replaced by  
\begin{enumerate}[(i)]
\item[(i)'] $\lim\limits_{d(x)\to 0}\dfrac{u(x)}{-\log (d(x))}=1$.
\end{enumerate}
\end{theorem}
\begin{remark} Lemma $\ref{lem:su}$ in the Appendix is a slightly stronger version of this theorem, where the convergence is shown to be uniform in $||g||_{L^\infty(\Omega)}$. This stronger version will be necessary in Section 5. 
\end{remark}
\noindent Lastly, we provide a stability result particular to the types of couplings we are studying. 

\begin{theorem}\label{thm:hjstab}
Assume $(F1)$ and fix $x_0\in \Omega$. Let $\mu_n$ be a sequence of probability measures on $\Omega$ which converge in $L^1(\Omega)$ to $\mu$, and $(u_n, \rho_n)\in W^{2,r}_{loc}(\Omega)\times \mathbb{R}$ be the unique explosive solution of 
\begin{eqnarray*}
\begin{cases}
-\Delta u_n+|D u_n|^q+\rho_n =F(x;\mu_n)\text{ in }\Omega,\\
\lim\limits_{d(x) \to 0} u_n(x)=\infty,
\end{cases}
\end{eqnarray*}
subject to the normalization condition $u_n(x_0)=0$. Then there exists $(u, \rho)\in W^{2,r}_{loc}(\Omega)\times \mathbb{R}$, such that $\rho_n \to \rho$ and $u_n \to u$ locally uniformly as $n\to \infty$. Moreover, $(u,\rho)$ is the 
unique explosive solution of 
\begin{eqnarray}\label{ustab}
\begin{cases}
-\Delta u+|D u|^q+\rho =F(x;\mu)\text { in }\Omega,\\
\lim\limits_{d(x) \to 0} u(x)=\infty, u(x_0)=0.
\end{cases}
\end{eqnarray}
\end{theorem}
\begin{proof}
It follows from $(F1)$ that $F(x;\mu_n)$ and $F(x; \mu)$ are bounded in $L^\infty(\Omega)$.  Theorem $\ref{thm: bounds}$ yields that $\rho_n$ is bounded and hence  converges along a subsequence to a constant $\rho$. The sequence $u_n$ is bounded in $W_{loc}^{2,r}(\Omega)$ so, by Arzel\`a-Ascoli, we can extract a subsequence converging locally uniformly to $u \in W^{2,r}_{loc}(\Omega)$ such that $D u_n \to D u$ locally uniformly. By pointwise convergence, $u(x_0)=0$. 
\\\\
That $u(x)\to \infty$ as $d(x)\to 0$ follows from the fact that the sequence of equations have a common explosive subsolution. When $1<q<2$, one such subsolution is of the form $(C_q-\epsilon) d(x)^{(q-2)/(q-1)}-C_\epsilon$, where $C_q=(q-1)^{(q-2)/(q-1)} (2-q)^{-1}$, $\epsilon>0$, and $C_\epsilon=C_\epsilon(||F(x; \mu)||_{L^\infty(\Omega; L^1(\Omega))}, \epsilon, C_q)$. When $q=2$ the corresponding subsolution is $-(1-\epsilon) \log(d(x))$. 
\\\\
Lastly, in view of $(F1)$, $F(x; \mu_n)$ converges uniformly to $F(x; \mu)$. It follows from classic viscosity theory that $(u, \rho)$ is a viscosity solution of $(\ref{ustab})$. As the explosive solution satisfying $u(x_0)=0$ is unique, the limit is indepedent of the subsequence. 
\end{proof}
\subsection{Definition and Main Result}
We finish this section with the definition of a solution to the system $(\ref{eq:mainintro})$ and the main existence and uniqueness result stated in the introduction. 
\begin{definition}
A triplet $(u, \rho, m)\in W^{2,r}_{loc}(\Omega) \times \mathbb{R} \times W^{1,r}_{loc}(\Omega)$ is a \emph{solution} of $(\ref{eq:mainintro})$ if $(u, \rho)$ is an explosive solution of 
\begin{eqnarray*}
\begin{cases}
-\Delta u+|D u|^q+\rho =F(x; m)\text { in }\Omega,\\
\lim\limits_{d(x) \to 0} u(x)=\infty, 
\end{cases}
\end{eqnarray*}
and $m$ is a proper weak solution of 
\begin{align*}
\Delta m+div\left(m q|D u|^{q-2}D u\right)=0 \text{ in }\Omega.
\end{align*}
\end{definition}
\noindent We are now able to state the main result of the paper.
\begin{theorem}\label{thm:mainmain}
Assume that $1<q\leq 2$ and either $(F1)$ or $(F2)$ holds. Then  $(\ref{eq:mainintro})$ has a solution for all $r>1$. Moreover, if $(F3)$ also holds, then the solution is unique. 
\end{theorem}

\section{Fixed point for the non-local case}
In this section we will use Schaefer's fixed point theorem in $\mathcal{W}_\gamma(\Omega)$ to prove the existence of a solution to $(\ref{eq:mainintro})$ when $(F1)$ holds. 
\\ 
\\
Let $2<\gamma<(2q-1)/(q-1)$. Define $T_1: \mathcal{W}_\gamma(\Omega) \to W^{1,r}_{loc}(\Omega)$ to be the map that sends measures $\mu$ to $q|D u|^{q-2}D u$, where $(u, \rho)$ is any explosive solution of 
\begin{eqnarray}
\begin{cases}\label{eq:hjmu}
-\Delta u+|D u|^q+\rho =F(x;{\mu}) \text{ in }\Omega,\\
\lim\limits_{d(x) \to 0} u(x)=\infty .
\end{cases}
\end{eqnarray}
It is immediate from Theorem $\ref{thm: hj}$ that $T_1$ is well defined as explosive solutions are unique up to a constant. \\\\
Define $T_2: T_1(\mathcal{W}_\gamma(\Omega)) \to \mathcal{W}_\gamma(\Omega)$ to be the map that sends $b$ in the image of $T_1$ to the proper weak solution $m$ of the associated continuity equation
\begin{eqnarray}\label{eq:fpb}
\Delta m+div(m b)=0 \text{ in } \Omega. 
\end{eqnarray}
In view of Lemma \ref{lemma:uni} and Theorem \ref{thm:asy}, $(\ref{eq:fpb})$ has a unique proper weak solution $m$ for $b$ in the image of $T_1$. 
\\\\
Finally define $T= T_2 \circ T_1$. In the next result, we apply Schaefer's fixed point theorem to $T$, thereby obtaining a solution of the MFG system $(\ref{eq:mainintro})$. 

\begin{theorem}\label{thm:fixed}
For $F$ satisfying $(F1)$ and $2<\gamma<(2q-1)/(q-1)$, $T:\mathcal{W}_\gamma(\Omega) \to \mathcal{W}_\gamma(\Omega)$ has a fixed point. 
\end{theorem}
\begin{proof}
Recall that to apply Schaefer's fixed point theorem, we need to show that $T$ is continuous, $T$ maps bounded sets to precompact sets, and if $F \subset W_\gamma$ is the set defined by
\begin{align}\label{eq:eff}
F=\{m\in \mathcal{W}_\gamma (\Omega)| m=\lambda T(m) \text{ for some }0\leq \lambda \leq 1\},
\end{align}
then $F$ is bounded.\\\\
We start by showing that $T$ is continuous. Consider a sequence $\mu_n$ that converges to $\mu$ in $\mathcal{W}_\gamma(\Omega)$. To simplify the notation, define $b_n:=T_1(\mu_n)$, $b:=T_1(\mu)$, and $m_n:=T_2(b_n)$. We first show that $b_n \to b$ locally uniformly. We then show that the sequence $m_n$ converges in $\mathcal{W}_\gamma(\Omega)$ to a measure $m$ and prove that $m=T_2(b)=T(\mu)$.\\\\
Fix $x_0 \in \Omega$ and let $(u_n, \rho_n)$ be the unique explosive solution of $(\ref{eq:hjmu})$ that corresponds to $F(x;{\mu}_n)$ and satisfies $u_n(x_0)=0$. Note that $T_1(\mu_n)= q|D u_n|^{q-2}D u_n$ by the remark following $(\ref{eq:hjmu})$. Since $\mu_n$ converges to $\mu$ in $L^1(\Omega)$, the local uniform convergence of $T_1(\mu_n)$ to $T_1(\mu)$ follows from Theorem $\ref{thm:hjstab}$. \\\\
In view of the local uniform bounds on $b_n$, Theorem \ref{thm:mdelta} implies the sequence of measures $m_n$ are uniformly bounded in $W^{1,r}_{loc}(\Omega)$ for all $r>1$ and, hence, locally uniformly bounded in $C^{0,\alpha}(\Omega)$ for $\alpha<1-n/r$. Along subsequences, the measures $m_n$ converge locally uniformly to a measure $m$. Consequently, we can pass to the limit in the weak formulation of $m_n$ and obtain that $m$ is a non-negative measure that satisfies the KFP equation with drift $b=T_1(\mu)$. It remains to show that $m_n$ converges to $m$ in $\mathcal{W}_\gamma(\Omega)$. 
\\\\
Lemma \ref{lemma:uni} yields that if  $(2q-1)/(q-1)>\gamma'>\gamma$, then the sequence $m_n$ is bounded in $\mathcal{W}_{\gamma'}(\Omega)$. This implies that the sequence $d^{-\gamma}m_n$ is uniformly tight. Indeed, for some constant $C$, independent of the choice of $\gamma$ and $\gamma'$,
\begin{align*}
C &\geq \int_{\Omega} d^{-\gamma'}m_n \geq \delta^{-\gamma'+\gamma}\left(\int_{\Omega\backslash \Omega_\delta}d^{-\gamma} m_n\right) 
\end{align*}
It then follows from the bound 
\begin{align*}
\int_{\Omega}d^{-\gamma}|m_n-m|\leq \int_{\Omega_\delta}d^{-\gamma}|m_n-m|+\int_{\Omega\backslash \Omega_\delta}d^{-\gamma} m_n+\int_{\Omega\backslash \Omega_\delta}d^{-\gamma} m,
\end{align*}
that $||m_n-m||_{\gamma,\Omega}\to 0$ as $n\to \infty$. 
\\\\
Since convergence in $\mathcal{W}_\gamma(\Omega)$ implies convergence in $L^1(\Omega)$, this also proves that $\int_\Omega m=1$. The uniqueness of $m$ then yields the convergence of the full sequence $m_n=T(\mu_n)$ to $T(\mu)$. 
\\\\
Next we show that $T$ maps bounded sets to precompact ones. Indeed let $\mu_n$ be a bounded sequence in $\mathcal{W}_\gamma(\Omega)$, and consider $u_n$, $b_n$, and $m_n$ defined as before. It suffices to show that there is a convergent subsequence of $m_n$. The sequence $\mu_n$ is bounded in $L^1(\Omega)$ and $F(x;\mu_n)$ is bounded in $L^\infty(\Omega; L^1(\Omega))$. Moreover, Theorem \ref{thm: bounds} yields that $b_n$ is locally uniformly bounded. Finally, Lemma \ref{thm:mdelta} implies that $m_n$ is uniformly bounded in $C^{0,\alpha}_{loc}(\Omega)$. Passing to subsequences, the measures $m_n$ converge locally uniformly to a measure $m$. Proceeding as in the continuity argument, we find that the sequence $d^{-\gamma}m_n$ is uniformly tight and $||m_n-m||_{\gamma,\Omega}\to 0$. 
\\\\
It remains to check that the set $F$ defined above is bounded. To see this, first observe that if $m=\lambda T(m)$, then $\int_\Omega m=\lambda\leq 1$. Hence $||F(x;m)||_{L^\infty(\Omega; L^1(\Omega))}$ is bounded by assumption $(F1)$. Next we apply Theorem \ref{thm: bounds} to find that $q|D u|^{q-2}D u$ is bounded on $\Omega_\delta$ by a constant $C(\delta, ||F(x;m)||_{L^\infty(\Omega; L^1(\Omega))})$. Lemma \ref{lemma:uni} then implies that $||m||_{\gamma,\Omega}$ is bounded. 
\end{proof}
\section{Fixed point argument for the local setting}\label{s:fixed}
In this section we will show existence of solutions to $(\ref{eq:mainintro})$ under assumption $(F2)$ by considering the system $(\ref{eq:deltap})$, introduced in Section 1, which approximates $(\ref{eq:mainintro})$. First, we will use Schaefer's fixed point theorem to show the system $(\ref{eq:deltap})$ has a solution. Then, using the uniform tightness results of Lemma $\ref{thm:mpro}$ and the local uniform bounds of Theorem $\ref{thm: bounds}$, we will show that solutions of $(\ref{eq:deltap})$ converge locally uniformly to solutions of $(\ref{eq:mainintro})$.
\subsection{The approximate problem}
We begin by making precise the notion of solution to the approximating system $(\ref{eq:deltap})$. We then show the existence of a solution using a fixed point theorem argument similar to that used in the non-local setting. 
\\\\
To utilize the asymptotic behavior of the explosive solutions to HJB equations at the boundary of $\Omega$ as well as the $W_{loc}^{1,r}$ bounds of Neumann weak solutions to KFP equations on $\Omega_\delta$, the KFP equation of $(\ref{eq:deltap})$ is solved in the subdomain $\Omega_\delta$ and the HJB equation is solved in $\Omega$. 


\begin{definition}\label{def: delta}
A triplet $(u_\delta, \rho_\delta, m_\delta)\in W^{2,r}_{loc}(\Omega)\times \mathbb{R}\times W^{1,r}(\Omega_\delta)$ is a \emph{solution} of $(\ref{eq:deltap})$ if $(u_\delta, \rho_\delta)$ is an explosive solution of 
\begin{eqnarray*}
\begin{cases}
-\Delta u_\delta+|D u_\delta|^q+\rho_\delta =F(x; {m}_\delta) \text{ in }\Omega,\\
\lim\limits_{d(x) \to 0} u_\delta(x)=\infty ,
\end{cases}
\end{eqnarray*}
and $m_\delta$ is a Neumann weak solution of 
\begin{align*}
\Delta m_\delta+div\left(m_\delta q|D u_\delta|^{q-2}D u_\delta\right)=0 \text{ in }\Omega_\delta.
\end{align*}
\end{definition}
\noindent When $F$ satisfies $(F2)$, existence of a solution to this system follows from Schaefer's fixed point theorem.  
\\\\
\noindent Indeed, let $T_1: C^{0,\alpha}(\Omega_\delta)\to W^{1,r}_{loc}(\Omega)$ be the map that sends measures $\mu$ to $q|D u|^{q-2}D u$, where $(u, \rho)$ is the explosive solution of 
\begin{eqnarray*}
\begin{cases}
-\Delta u+|D u|^q+\rho =F(x;{\mu}) \text{ in }\Omega\\
\lim\limits_{d(x) \to 0} u(x)=\infty,
\end{cases}
\end{eqnarray*}
which is well defined for all $\alpha>0$ and $1<r<\infty$. It follows from Theorem $\ref{thm: hj}$ that the system has a solution $u\in W^{2,r}_{loc}(\Omega)$ for every $r>1$, and $D u$ is unique. \\\\
Let $T_2: W^{1,r}_{loc}(\Omega) \to C^{0,\alpha}(\Omega_\delta)$ be the map that sends $b$ to the Neumann weak solution $m_\delta$ of
\begin{eqnarray*}
\begin{cases}
\Delta m_\delta+div(m_\delta b_\delta)=0 &\text{ in } \Omega_\delta\\
(D m_\delta +m_\delta b_\delta)\cdot\nu=0 &\text{ on }\partial \Omega_\delta,
\end{cases}
\end{eqnarray*}
where $b_\delta$ is the restriction of $b$ to the domain $\Omega_\delta$, which is also well defined. In view of Theorem \ref{thm: hj}, there is a unique Neumann weak solution $m_\delta$. Moreover, $m_\delta$ is in $W^{1,s}(\Omega_\delta)$ for all $1<s<\infty$ and, by Sobelov embedding, is in $C^{0, \alpha}(\Omega_\delta)$ for all $0<\alpha<1$.  Composing the two, we define $T= T_2 \circ T_1$. 

\begin{theorem}\label{thm:fixedpoint}
For $F$ satisfying $(F2)$ and $\alpha>0$, $T:C^{0,\alpha}(\Omega_\delta) \to C^{0,\alpha}(\Omega_\delta)$ has a fixed point. 
\end{theorem}
\begin{proof}
We apply Schaefer's fixed point theorem as in Theorem \ref{thm:fixed}. As we will be working with an HJB equation set in $\Omega$ and a KFP equation set in $\Omega_\delta$, we will be exploiting that the method for extending measures from $\Omega_\delta$ to $\Omega$ defined in $(\ref{tilde})$ is $C^{0,\alpha}$-norm preserving.
\\\\
To show the continuity of $T$, we consider a sequence $\mu_n$ which converges to $\mu$ in $C^{0,\alpha}(\Omega_\delta)$, and we start by showing that $T_1(\mu_n)$ converges to $T_1(\mu)$ locally uniformly. To simplify notation, we let $b_n:=T_1(\mu_n)$, $b:=T_1(\mu)$, and $m_n:=T_2(b_n)$.\\\\
Let $(u_n, \rho_n)$ be the unique explosive solution of $(\ref{eq:hj})$ that corresponds to $F(x;{\mu}_n)$ and satisfies $u_n(x_0)=0$, where $x_0 \in \Omega$ is fixed. As the sequence $\mu_n$ converges to $\mu$ in $C^{0, \alpha}(\Omega_\delta)$, $\mu_n$ converges uniformly  to $\mu$ in $\Omega_\delta$, and as a result $\tilde{\mu}_n$ converges uniformly to $\tilde{\mu}$ in $\Omega$, where $\tilde{\mu}_n$ and $\tilde{\mu}$ are defined as in $(\ref{tilde})$. \\\\
Due to the stability properties of viscosity solutions, to prove continuity of $T_1$ under assumption $(F2)$, it is sufficient to show that if $x_n\to x$ in $\Omega$, then $f(\tilde{\mu}_n(x_n))\to f(\tilde{\mu}(x))$. However, this follows immediately from the continuity of $f$ and the equicontinuity of $\tilde{\mu}_n$. 
\\\\
Next we show that $T_2$ is continuous. Since $||b_n||_{L_{loc}^{\infty}(\Omega)}$ is bounded in $n$, $m_n$ is uniformly bounded in $W^{1,r}(\Omega_\delta)$ and, hence, bounded in $C^{0,\beta}(\Omega_\delta)$ for $\alpha<\beta<1-n/r$. In view of the compact embedding of  $C^{0,\beta}(\Omega_\delta)$ in $C^{0,\alpha}(\Omega_\delta)$, there is a subsequence of $m_n$ that converges in $C^{0, \alpha}(\Omega_\delta)$  to some $m\in C^{0, \alpha}(\Omega_\delta)$. Moreover, $D m_n$ converges weakly to $D m$. Consequently, we can pass to the limit in the weak formulation of $m_n$ and obtain that $m\in W^{1,r}(\Omega_\delta)$ is a positive probability measure that satisfies the KFP equation with drift $b=T_1(\mu)$. The uniqueness of $m$ yields convergence of $T(\mu_n)$ to $T(\mu)$ along the full sequence. \\\\
Next, to show that $T$ is precompact, we will use the bounds in Theorem $\ref{thm: bounds}$, Lemma $\ref{thm:mdelta}$, and the Sobelov embedding theorem. \\\\
For $\mu \in C^{0,\alpha}(\Omega_\delta)$ and $b=T_1(\mu)$, Theorem \ref{thm: bounds} yields $|b(x)|\leq C d(x)^{-1}$, where $C$ depends on $\Omega$ and $||F(x;{\mu})||_{L^{\infty}(\Omega; L^1(\Omega))}$. Note that by assumption $(F2)$, $||b||_{L^\infty(\Omega_\delta)}$ is bounded by a constant depending on $||f||_{{L^\infty}(\mathbb{R})}$. \\\\
Lemma $\ref{thm:mdelta}$ yields that $||T(\mu)||_{W^{1,r}(\Omega_\delta)}$ is bounded by a constant depending on $||f||_{L^\infty(\mathbb{R})}$.
Since $W^{1,r}(\Omega_\delta)$ embeds compactly in $C^{0,\alpha}(\Omega_\delta)$ for $\alpha <1-n/r$, choosing appropriately large $r$, $T$ sends bounded sets of $C^{0, \alpha}(\Omega_\delta)$ to precompact ones.\\\\
It remains to check that if $F$ is the set defined by
\begin{align*}
F=\{m\in C^{0,\alpha}(\Omega_\delta) | m=\lambda T(m) \text{ for some }0\leq \lambda \leq 1\},
\end{align*}
then $F$ is bounded.
This follows from Theorem \ref{thm:mdelta} as $q|D u|^{q-2}D u$ is bounded in $L^\infty(\Omega_\delta)$.

\end{proof}

\subsection{Limit of the approximate problem}\label{s:lim}
In this section we prove that in the local setting, the solutions $(u_\delta, \rho, m_\delta)$ of $(\ref{eq:deltap})$ converge locally uniformly to solutions of $(\ref{eq:mainintro})$. We begin by showing that, along subsequences, $u_\delta$ has a local uniform limit $u$ using the Arzel\`a-Ascoli Theorem. We then construct a Lyapanov function and use Lemma \ref{weakm} and Lemma \ref{thm:mpro} to show that the measures $m_\delta$ converge locally uniformly to a proper weak solution $m$ of a continuity equation. Lastly we show that $u$ is the explosive solution corresponding to $f(m(x))$. 
\begin{theorem}\label{thm:delta1}
Given $\delta>0$ and $F$ satisfying $(F2)$, let $x_0\in \Omega$, and $(u_\delta, \rho_\delta,m_\delta) \in W_{loc}^{2,r}(\Omega)\times \mathbb{R}\times W^{1,r}(\Omega_\delta)$  be a solution of $(\ref{eq:deltap})$ satisfying $u_\delta (x_0)=0$. Along subsequences, $\rho_\delta$ converges to a constant $\rho$ and $u_\delta$ converges in $C^1_{loc}(\Omega)$ to $u \in W_{loc}^{2,r}(\Omega)$.
\end{theorem}
\begin{proof}
It follows from Theorem $\ref{thm: bounds}$, and the uniform in $\delta$ bounds on $||F(x;m_\delta)||_{L^\infty(\Omega; L^1(\Omega))}$, that $\rho_\delta$ and the $W^{1,r}(\Omega_\delta)$ bounds of $u_\delta$ are bounded uniformly in $\delta$. The desired convergence along subsequences follows by the Arzel\`a-Ascoli Theorem.
\end{proof}
\noindent For now we make no claims about the equation $u$ solves. We must first make sense of the limit of the measures $m_\delta$. 
\begin{theorem} Given $\delta>0$ and $F$ satisfying $(F2)$, let $(u_\delta, \rho_\delta,m_\delta) \in W_{loc}^{2,r}(\Omega)\times \mathbb{R}\times W^{1,r}(\Omega_\delta)$  be a solution of $(\ref{eq:deltap})$. Along subsequences, $m_\delta$ converges locally uniformly  to $m \in W^{1,r}_{loc}(\Omega)$ which is a proper weak solution of  
\begin{eqnarray*}
\Delta m+div(m b)=0 \text{ in }\Omega,
\end{eqnarray*}
where $b=q|D u|^{q-2} D u$ and $u$ is the limit function from Theorem $\ref{thm:delta1}$. 
\end{theorem}
\begin{proof}
In view of Lemma \ref{weakm} and Lemma \ref{thm:mpro}, it suffices to construct a function $V$ satisfying $(\ref{v1})$ and $(\ref{v2})$ for $b_\delta=q|D u_\delta|^{q-2} D u_\delta$. 
\\\\
Consider $V(x)= d(x)^{-C+1+\epsilon}$, where $C=q/(q-1)$ and $0<\epsilon< C-1$, which satisfies the desired growth condition $(\ref{v1})$ at the boundary. The $L^\infty (\Omega; L^1(\Omega))$-norm of the coupling term $F(x;{m}_\delta)$ is uniformly bounded in $\delta$ by $||f||_{L^\infty(\mathbb{R})}$. By Lemma $\ref{lem:su}$, $b_\delta(x)\cdot \nu(x) d(x)\to C$ uniformly in $\delta$ as $d(x)\to 0$. It follows that the limit $\Delta V -b_\delta \cdot D V \to -\infty$ at the boundary is likewise uniform in $\delta$, and $(\ref{v2})$ is satisfied.  The uniform tightness of the measures $m_\delta$ ensures that $\int_\Omega m=1$ and $m_\delta$ converges to $m$ in $L^1(\Omega)$. 
\end{proof}
\noindent We are now ready to state the equation that $u$ solves.

\begin{theorem}
Let $F$ satisfy $(F2)$ and $(u_\delta, \rho_\delta,m_\delta) \in W_{loc}^{2,r}(\Omega)\times \mathbb{R} \times W^{1,r}(\Omega_\delta)$  be a solution of $(\ref{eq:deltap})$. Let $u$, $m$ and $\rho$ be as constructed in Theorems 5.1 and 5.2. Then  $(u, \rho)$ is the explosive solution of 
\begin{align*}
\begin{cases}
-\Delta u+|D u|^q+\rho =F(x;m) \text{ in }\Omega,\\
\lim\limits_{d(x) \to 0} u(x)=\infty.
\end{cases}
\end{align*}
\end{theorem}
\begin{proof}
The proof is very similar to the stability argument in Theorem \ref{thm:fixedpoint}, and follows from the local uniform H\"older bounds on $m_\delta$. Indeed, consider a sequence $x_\delta\to x$ contained in a compact set $K$, and let $M=\sup_\delta||m_\delta||_{C^{0,\alpha}(K)}$. Then, 
\begin{align}
|m_\delta(x_\delta)-m(x)|&\leq |m_\delta(x)-m(x)|+|m_\delta(x_\delta)-m_\delta(x)| \nonumber\\
&\leq |m_\delta(x)-m(x)|+M|x_\delta-x|^{\alpha}.\label{eq:expression}
\end{align}
By the uniform convergence of $m_\delta$ to $m$ on $K$, $(\ref{eq:expression})$ goes to $0$ as $\delta \to 0$. The continuity of $f$ allows us to pass to the limit in the HJB equation, obtaining that $u$ is a viscosity solution of 
\begin{align*}
-\Delta u+|D u|^q+\rho =F(x;m) \text{ in }\Omega.
\end{align*}
The argument that $u$ blows up at the boundary is the same as in Theorem $\ref{thm:hjstab}$. 
\end{proof}

%

%
%
%

\section{Uniqueness}
We are now ready to address the uniqueness claim in Theorem $\ref{thm:mainmain}$. Throughout the discussion we take $1<q<2$ and deal with the case $q=2$ at the end of the section. 
\begin{proof}[Proof of Theorem $\ref{thm:mainmain}$]
It only remains to show that uniqueness holds when $F$ satisfies $(F3)$. Here we follow the classical Lasry-Lions uniqueness argument. Consider two solutions $(u_1,\rho_1, m_1), (u_2, \rho_2, m_2) \in W^{2,r}_{loc}(\Omega)\times \mathbb{R}\times W^{1,r}_{loc}(\Omega)$ of $(\ref{eq:mainintro})$. Without loss of generality, we may assume that $u_1(x_0)=u_2(x_0)=0$ for a fixed $x_0\in \Omega$. 
\\\\
Letting $\overline{u}:=u_1-u_2$, $\overline{\rho}:=\rho_1-\rho_2$, and $\overline{m}:=m_1-m_2$, it follows that $(\overline{u},\overline{\rho},\overline{m})$ satisfies the system
\begin{align*}
\begin{cases}
-\Delta \overline{u}+|D u_1|^q-|D u_2|^q+\overline{\rho} =F(x;m_1)-F(x;m_2)&\text{ in }\Omega,\\
\Delta \overline{m}+div(m_1 b_1-m_2 b_2)=0 &\text{ in }\Omega,
\end{cases}
\end{align*}
where $b_1:=q|D u_1|^{q-2}D u_1$ and $b_2:=q|D u_2|^{q-2}D u_2$. \\\\
Consider a non-negative and compactly supported function $\phi \in C^{\infty}(\Omega)$; the particular choice of $\phi$ will be made later. \\\\
Multiplying the first equation by $\overline{m}\phi$ and integrating by parts, we obtain
\begin{align*}
\int_{\Omega}\left[(D \overline{u})\cdot D(\overline{m}\phi)+(|D u_1|^q-|D u_1|^q)\overline{m}\phi+(\overline{\rho})\overline{m}\phi\right] =\int_\Omega (F(x;m_1)-F(x;m_2))\overline{m}\phi.
\end{align*}
Similarly, using $\overline{u}\phi$ as a test function in the KFP equation, we obtain
\begin{align*}
\int_{\Omega} \left[-D (\overline{u}\phi)\cdot (D \overline{m})-(m_1 b_1-m_2 b_2)\cdot D(\overline{u}\phi)\right]=0.
\end{align*}
\noindent Adding the equations and regrouping gives:
\begin{align}
\int_\Omega (F(x;m_1)-F(x;m_2))\overline{m}\phi&=\int_{\Omega}\phi \left[(|D u_1|^q-|D u_1|^q)\overline{m}-(m_1 b_1-m_2 b_2)\cdot D(\overline{u})\right]\nonumber\\
&+\int_{\Omega}(\overline{\rho})\overline{m}\phi \nonumber\\
&-\int_{\Omega}(D \phi)\cdot [(m_1 b_1-m_2 b_2)\overline{u}]\nonumber\\
&+\int_{\Omega} (D\phi) \cdot (\overline{m}D\overline{u}-\overline{u}D\overline{m})\nonumber\\
&=\int_{\Omega}\phi m_1(|D u_1|^q-|D u_2|^q-b_1\cdot D(u_1-u_2))\label{con1}\\
&+\int_\Omega \phi m_2(|D u_2|^q-|D u_1|^q- b_2 \cdot D (u_2-u_1))\label{con2}\\
&+\int_{\Omega}(\overline{\rho})\overline{m}\phi \nonumber\\
&-\int_{\Omega}(D \phi)\cdot [(m_1 b_1-m_2 b_2)\overline{u}+2\overline{m}D\overline{u}].\nonumber\\
&+\int_{\Omega} \overline{m}\Delta \phi \overline{u} \nonumber
\end{align}
The left-hand side is positive by the monotonicity assumption $(F3)$, and $(\ref{con1})$ and $(\ref{con2})$ are negative by the convexity of the Hamiltonian. We will show that the remaining terms on the right hand side can be made arbitrarily small for the right choice of $\phi$. \\
\\
We are now ready to define $\phi$. Let $\psi(s):[0,\infty) \to [0,1]$ be a smooth function that is $0$  in a neighborhood of $[0,1]$, $1$ in a neighborhood of $2$, and constant for $s\geq 2$. For $\delta <\min(1,\epsilon_0)/2$, let $\phi(x)=\psi(d(x)/\delta)$. There exists a constant $C_1=C_1(\Omega, \psi)>0$ such that, in $\Omega_{\delta}\backslash \Omega_{2\delta}$,
\begin{align}\label{eq:dpsi}
|D \phi|\leq \dfrac{1}{\delta}||\psi'||_{L^\infty([0,\infty))}||D d||_{L^\infty(\Omega_{\delta}\backslash \Omega_{2\delta})}\leq C_1 \dfrac{1}{\delta}.
\end{align}
 and 
 \begin{align*}
 |\Delta \phi|\leq \dfrac{1}{\delta^2}||\psi''||_{L^\infty([0,\infty))}||D d||_{L^\infty(\Omega_{\delta}\backslash \Omega_{2\delta})}^2+\dfrac{1}{\delta}||\psi'||_{L^\infty([0,\infty))}||\Delta d||_{L^\infty(\Omega_{\delta}\backslash \Omega_{2\delta})}\leq C_1 \dfrac{1}{\delta^2}.
 \end{align*}
Together with the local bounds on $|u_i|$ and $|Du_i|$ from Theorem $\ref{thm: bounds}$, it follows from $(\ref{eq:dpsi})$ that there exists a constant $C_2>0$, independent of $\delta$, for which
\begin{align}
\left|\int_{\Omega}(D \phi)\cdot [(m_1 b_1-m_2 b_2)\overline{u}+2\overline{m}D\overline{u}]\right|&\leq C_2\int_{\Omega_{\delta}\backslash \Omega_{2\delta}}\dfrac{1}{\delta}(d^{-1/(q-1)})^{q-1}(m_1+m_2)d^{(q-2)/(q-1)}\nonumber\\
&+C_2\int_{\Omega_{\delta}\backslash \Omega_{2\delta}}\dfrac{1}{\delta} d^{-1/(q-1)}|\overline{m}|\nonumber\\
&\leq 2C_2\int_{\Omega_{\delta}\backslash \Omega_{2\delta}}d^{-q/(q-1)}(m_1+m_2+|\overline{m}|)\label{b2}.
\end{align}
Similarly there exists a constant $C_3>0$, independent of $\delta$, for which 
\begin{align}
\left|\int_{\Omega}(\Delta \phi)(\overline{u})\overline{m}\right|&\leq C_3\int_{\Omega_{\delta}\backslash \Omega_{2\delta}}\dfrac{1}{\delta^2}(m_1+m_2)d^{(q-2)/(q-1)}\\
&\leq 4C_3\int_{\Omega_{\delta}\backslash \Omega_{2\delta}}d^{-q/(q-1)}(m_1+m_2)\label{b3}.
\end{align}
We recall from Lemma \ref{lemma:uni} that, for small $\epsilon$, $\int_\Omega d(x)^{-q/(q-1)-1+\epsilon}m(x)$  is finite. It follows that $(\ref{b2})$ and $(\ref{b3})$ can be made arbitrarily small for $\delta$ small enough. 
\\\\
To bound the remaining term, we use that $m_1+m_2\in L^1(\Omega)$ and $\int_{\Omega}(m_1-m_2)=0$, obtaining
\begin{align*}
\left|\int_{\Omega}(\overline{\rho})\overline{m}\phi\right|&\leq |\overline{\rho}|\int_{\Omega_{\delta}\backslash \Omega_{2\delta}}|m_1-m_2|+|\overline{\rho}|\left|\int_{\Omega_{2\delta}}m_1-m_2\right|\\
&\leq |\overline{\rho}|\left(\int_{\Omega\backslash \Omega_{2\delta}}(m_1+m_2)+\left|0-\int_{\Omega\backslash \Omega_{2\delta}}(m_1-m_2)\right|\right)\\
&\leq 2|\overline{\rho}|\int_{\Omega\backslash \Omega_{2\delta}}(m_1+m_2),
\end{align*}
which can also be made arbitrarily small for $\delta$ small enough. \\\\
It follows from the monotonicity of $F(x;m)$, that $m_1=m_2$. That $u_1=u_2$ and $\rho_1=\rho_2$ follows from the uniqueness of solutions to $(\ref{eq:hj})$.  
\end{proof}

\begin{remark}
To prove uniqueness of solutions to $(\ref{eq:mainintro})$ when $q=2$, $(\ref{b2})$ should be replaced by an estimate of the form 
\begin{align}
C \int_{\Omega_{\delta}\backslash \Omega_{2\delta}}d^{-1}(-\log(d))(m_1+m_2),
\end{align}
where  $C>0$ is a constant independent of $\delta$. As $d^{-1}(-\log(d))$ has growth slower than $d^{-2+\epsilon}$, the rest of the proof then proceeds as in the $1<q<2$ case. 
\end{remark}

\section{Appendix}
In this section we prove Lemma \ref{lem:su}, a stronger version of $(ii)$ in Theorem \ref{thm:asy}, and Lemma \ref{lemma:uni}. 
\begin{lemma}\label{lem:su}
Let $m\in W^{1,r}_{loc}(\Omega)$ and consider the unique solution $(u_m,\rho_m)\in W^{2,r}_{loc}(\Omega)\times\mathbb{R}$ of 
\begin{align}
\begin{cases}
-\Delta u_m+|D u_m|^q+\rho_m =F(x;m) \text{ in }\Omega,\\
\lim\limits_{d(x) \to 0} u_m(x)=\infty.
\end{cases}
\end{align}
Let $b_m:=q|D u_m|^{q-2}D u_m$. Then, uniformly in $||F(x ; m)||_{L^\infty(\Omega; L^1(\Omega))}$,
\begin{equation*}
\lim_{d(x)\to 0}(b_m(x)\cdot \nu(x))d(x)=\dfrac{q}{q-1}.
\end{equation*}

\end{lemma}

\noindent It will be enough to show uniform convergence of $(D u_m(x)\cdot \nu(x))d(x)^{1/(q-1)}$ as $d(x)\to 0$. \\
We prove this by following the proof of Theorem 2.3 in $\cite{por}$ and tracking how the convergence depends on $||F(x ;m)||_{L^\infty(\Omega; L^1(\Omega))}$. 
\\\\
\noindent In the proof that follows, we argue first when $1<q<2$, and discuss the $q=2$ case at the end of the proof. The same arguments also hold for Hamiltonians satisfying $(\ref{eq:gen})$ locally uniformly in $p$ as $\delta \to 0$. 

\begin{proof}[Proof of Lemma \ref{lem:su}] Since $\Omega$ has a uniform $C^2$ boundary, by Theorem \ref{thm:asy} there exists a $\delta_0>0$, depending on $||F(x ;m)||_{L^\infty(\Omega; L^1(\Omega))}$, such that for every $x\in \partial \Omega$, the ball $B_{\delta_0}(x-\delta_0 \nu(x))$ is contained in $\Omega$, and, for $d(x)<\delta_0$ ,
\begin{align}\label{d_0}
\left|\dfrac{u_m(x)}{d(x)^{(q-2)/(q-1)}}-C_q\right|\leq \dfrac{C_q}{2},
\end{align}
where $C_q=(q-1)^{(q-2)/(q-1)} (2-q)^{-1}$. \\\\
We begin by fixing $x_0\in \partial\Omega$ and a system of coordinates $(\eta_1, \ldots, \eta_n)$ centered at $x_0$ whose $\eta_1$-axis is the inner normal direction. Although we have fixed an $x_0$, any bounds and rates that follow will not depend on $x_0\in \partial\Omega$.\\\\
Let $O_\delta=(\delta,\ldots, 0)$ and consider the domain 
\begin{align*}
D_\delta=B(O_\delta, \delta^{1-\sigma})\cap B((\delta_0, 0, \ldots, 0), \delta_0-\delta)\text{ , }\sigma \in (0, 1/2).
\end{align*} 
Applying the change of variable $\zeta=(\eta-O_\delta)/\delta$, the domain $\tilde{D}_\delta:=\{(\eta-O_\delta)/\delta: \eta \in D_\delta\}$ converges to the half space $\{\zeta: \zeta_1>0\}$, where $\zeta_1$ is the component of the vector in the $-\nu(x)$ direction. Under this change of variable we have the bounds  
\begin{align}\label{eta}
(\zeta_1+1)\delta\leq d(\eta)\leq (\zeta_1+1)\delta+O(\delta^{2-2\sigma}).
\end{align}

\noindent It is now possible to introduce, for $\delta<\delta_0$, the blow up function 
\begin{align*}
v_{\delta,m}(\zeta)=\dfrac{u_m(\delta \zeta+O_\delta)}{\delta^{(q-2)/(q-1)}},
\end{align*}
with domain $\tilde{D}_\delta$, which, in view of $(\ref{d_0})$, is uniformly bounded in $\delta, ||F(x ;m)||_{L^\infty(\Omega; L^1(\Omega))},$ and $\zeta$. 
\\\\
Straightforward computations show that $v_{\delta,m}$ describes the blow-up of $u_m$ near the boundary. Indeed, for $\eta_1=\delta \zeta_1+O_\delta$,
\begin{align*}D v_{\delta, m}(\zeta_1)=\dfrac{D u_m(\delta \zeta_1 +O_\delta)}{C_q \delta^{-1/(q-1)}}=\dfrac{D u_m(\eta_1)}{C_q d(\eta_1)^{-1/(q-1)}}(1+\zeta_1)^{-1/(q-1)}.
\end{align*}

\noindent In $\tilde{D}_\delta$, $v_{\delta,m}$ solves, in the viscosity sense,
\begin{align}\label{vd}
-\Delta v_{\delta, m}+\dfrac{\rho_m \delta^{q/(q-1)}}{C_q}+C_q^{q-1}|D v_{\delta,m}|^q=\dfrac{F(\delta \zeta+O_\delta; m)\delta^{q/(q-1)}}{C_q}.  
\end{align}

\noindent Moreover, $D v_{\delta,m}(\zeta)$ satisfies the estimate
\begin{align*}
|D v_{\delta,m}(\zeta)|=\dfrac{|D u_m(\delta \zeta +O_\delta)|}{C_q \delta^{-1/(q-1)}}\leq C_m d(\eta)^{-1/(q-1)}\delta^{1/(q-1)}\leq C_m,
\end{align*}
where $C_m= C_m(||F(x ;m)||_{L^\infty(\Omega; L^1(\Omega))})$ comes from the locally uniform bound on $Du$ from Theorem \ref{thm: bounds}. 
\\\\
\noindent Consider a family of $\{v_{\delta, m}\}$ with $\delta\to 0$ and measures $m$ such that $||F(x ;m)||_{L^\infty(\Omega; L^1(\Omega))}$ is uniformly bounded. This is the case, for example, if $(F1)$ or $(F2)$ are satisfied and the measures $m$ are uniformly bounded in $L^1(\Omega)$. It follows from elliptic regularity that $\{v_{\delta, m}\}$ has a convergent subsequence converging in $C^{1}_{loc}$ to $v$ defined in the half-plane $\{\zeta: \zeta_1>0\}$. \\\\
Since $\dfrac{\rho_{m} \delta^{q/(q-1)}}{C_q}\to 0$ and $\dfrac{F(\delta \zeta+O_\delta; m)\delta^{q/(q-1)}}{C_q}\to 0$ uniformly, $v$ is a viscosity solution of 
\begin{align}\label{v}
-\Delta v+C_q^{q-1}|D v|^q=0 \text{ in }\{\zeta: \zeta_1>0\}.
\end{align}
\noindent In addition, $v$ satisfies the boundary conditions
\begin{align*}
\lim_{\zeta_1\to 0^{+}}v(\zeta)=C_q \text{ and } \lim_{\zeta_1\to \infty}v(\zeta)=0. 
\end{align*}
The boundary conditions follow from 
\begin{align}\label{bounds}
v_{\delta,m}(\zeta)=\dfrac{u_m(\eta)}{d(\eta)^{(q-2)/(q-1)}}\dfrac{d(\eta)^{(q-2)/(q-1)}}{\delta^{(q-2)/(q-1)}},
\end{align}
and the estimates in $(\ref{eta})$.  As shown in Theorem 4.1 of \cite{por}, the equation $(\ref{v})$ paired with boundary conditions $(\ref{bounds})$ has a unique positive solution in the half-plane, $v(\zeta)=C_q(1+\zeta_1)^{(q-2)/(q-1)}$. From $(\ref{d_0})$, $v$ is indeed positive. Thus $D v_{\delta,m}$ converges locally uniformly to $C_q(1+\zeta_1)^{-1/(q-1)}(2-q)/(q-1)\nu(x_0)$, and $(Du_m(x)\cdot \nu(x))d(x)^{1/(q-1)}$ converges uniformly in $||F(x; m)||_{L^\infty(\Omega; L^1(\Omega))}$, as $d(x)\to 0$. 
\\\\
When $q=2$, rescaling $u_m$ as before does not yield an equation of the type $(\ref{vd})$. To circumvent this, we consider a function of the form $u_m(\eta)+\log(d(\eta))$, which is uniformly bounded in $||F(x;m)||_{L^\infty(\Omega; L^1(\Omega))}$ and has a limiting equation of the desired form. 
\\\\
\noindent For the case $q=2$, consider a coordinate system centered at $x_0\in \partial\Omega$, denoted by $O$, with the $\eta_1$-axis as the inner normal direction. By the inner sphere property, there is $\delta_0>0$ such that for every $x\in \partial \Omega$, the ball $B_{\delta_0}(x-\delta_0 \nu(x))$ is contained in $\Omega$. Under this new coordinate system define
\begin{align*}
D_\delta=B(O, \delta^{1-\sigma})\cap B((\delta_0, 0, \ldots, 0), \delta_0)\text{ , }\sigma \in (0, 1/2). 
\end{align*} 
Note that we changed the definition of $D_\delta$ here.  \\\\
Restricting $u_m$ to $D_\delta$, the function
\begin{align*}
v_{\delta,m}(\zeta)=u_m(\delta \zeta)+\log(\delta),
\end{align*}
is defined on the domain $\tilde{D}_\delta:=\{\eta/\delta: \eta \in D_\delta\}$. As before, $\tilde{D}_\delta$ converges to the half space $\{\zeta: \zeta_1>0\}$, where $\zeta_1$ is the component of the vector in the $-\nu(x)$ direction. 
Under this change of variable we have the bounds  
\begin{align*}
\zeta_1\delta\leq d(\eta)\leq \zeta_1\delta+O(\delta^{2-2\sigma}).
\end{align*}

\noindent In $\tilde{D}_\delta$, $v_{\delta,m}$ satisfies, in the viscosity sense,
\begin{align*}
-\Delta v_{\delta, m}+\rho_m \delta^2+|D v_{\delta,m}|^2=F(\delta \zeta; m)\delta^2.  
\end{align*} 
By Theorem II.3 of \cite{pll},  $u_m(\delta \zeta)+\log(d(\eta))$ is bounded by a constant $C$ which is uniform in $||F(x;m)||_{L^\infty(\Omega; L^1(\Omega))}$. From the bounds $(\ref{eta})$ on $d(\eta)$ it follows that $v_{\delta,m}$ is bounded locally uniformly. From standard elliptic theory, $v_\delta$ is relatively compact in $C^1_{loc}$. Thus, along subsequences $v_{\delta,m}$ converges in $C^1_{loc}$ to a function $v$, defined in the upper-half plane, which is a viscosity solution of
\begin{align*}
-\Delta v+|D v|^2=0.
\end{align*}
Moreover, for constant $C$ as defined above, 
\begin{align*}
v_{\delta, m}&= u_m(\delta \zeta)+\log(d(\eta))+\log\left(\dfrac{\delta}{d(\eta)}\right)\\
&\geq -C+\log\left(\dfrac{1}{\zeta_1+O(\delta^{1-2\sigma})}\right).
\end{align*} 
Thus we have the boundary condition 
\begin{align*}
\lim\limits_{\zeta_1\to 0^+} v(\zeta)=\infty.
\end{align*}
Solutions of the system 
\begin{align*}
\begin{cases}
-\Delta v+|D v|^2=0 \text{ in } \{\zeta \in \mathbb{R}^n\, |\, \zeta_1>0\},\\
\lim\limits_{\zeta_1\to 0^+} v(\zeta)=\infty,
\end{cases}
\end{align*}
are of the form $-\log \zeta_1+ \kappa $ for a constant $\kappa$, which follows from the change of variable $w=e^{-v}$. Thus $D v_{\delta,m}$ converges locally uniformly to $(1/\zeta_1)\nu(x_0)$, and $(Du_m(x)\cdot \nu(x))d(x)$ converges uniformly in $||F(x; m)||_{L^\infty(\Omega; L^1(\Omega))}$, as $d(x)\to 0$.

\end{proof}

We now prove Lemma \ref{lemma:uni} from Section 3.1.

{\newtheorem*{lemma*}{Lemma \ref{lemma:uni}}}
\begin{lemma*}
Let $b: \Omega\to \mathbb{R}^d$ satisfy $(\ref{eq:asym})$ with $C>1$.  Then there exists a unique proper weak solution $m\in W^{1,r}_{loc}(\Omega)$ of   
\begin{align}\label{eq:mapp}
\Delta m+div(m b)=0 \text{ in } \Omega,
\end{align}
and, for every $\epsilon \in (0,C-1)$, there exists a $\delta =\delta(C, \epsilon)$, such that
\begin{align}
\int_{\Omega} d^{-C-1+\epsilon}m\leq \hat{C}=\hat{C}(C, ||b||_{L^\infty(\Omega_\delta)}).\label{eq:wg}
\end{align}

\end{lemma*}
\begin{proof}
The existence of a proper weak solution $m$ was proven in Section \ref{s:fp}. To prove (\ref{eq:wg}), we mimick the proof of Lemma 1.1 in \cite{rus1}. 
\\\\
Let $V(x)= d(x)^{-C+1+\epsilon}$ for $0<\epsilon<C-1$. In the neighborhood $\Omega\backslash\Omega_{\epsilon_0}$ of $\partial \Omega$,
\begin{align*}
\Delta V -b \cdot D V &= (-C+1+\epsilon)d^{-C-1+\epsilon}[(-C+\epsilon) |D d|^2+d\Delta d-(b\cdot D d)d]\\
&=(-C+1+\epsilon)d^{-C-1+\epsilon}[-C+\epsilon+d\Delta d-(b\cdot D d)d].
\end{align*}
Since the last factor approaches $\epsilon$ as $d\to 0$, in a neighborhood of $\partial \Omega$, we have
\begin{align}\label{eq:Vass}
 \Delta V -b \cdot D V\lesssim -d^{-C-1+\epsilon}.
\end{align} 
Let $\delta_0>0$ be such that, for $x\in \Omega/\Omega_{\delta_0}$,
\begin{align*}
\Delta V -b \cdot D V \leq -1,
\end{align*} 
and define $S:=\int_{\Omega_{\delta_0}}|\Delta V -b \cdot D V|m$.
\\\\
Note that since $m$ is a probability measure on $\Omega$, $S$ is bounded by the $L^\infty$ norm of $|\Delta V -b \cdot D V|$ on $\Omega_{\delta_0}$.\\\\
We claim that, for all $\delta<\delta_0$,
\begin{equation}\label{eq:V}
\int_{\Omega_{\delta}}|\Delta V -b \cdot D V|m\leq 2S.
\end{equation}
Since $S$ is independent of $\delta$, it would then follow that 
\begin{align*}
\int_{\Omega}|\Delta V -b \cdot D V|m\leq 2S,
\end{align*}
and the growth of $\Delta V -b \cdot D V$ in $(\ref{eq:Vass})$ completes the proof of the lemma. \\\\
To prove $(\ref{eq:V})$, fix $\delta$, $s$, and $r$, such that $0<\delta_0^{-C+1+\epsilon}<s<\delta^{-C+1+\epsilon}<r$. Consider a nondecreasing cut-off function $\phi \in C^2(\mathbb{R})$ such that $\phi''\leq 0$, $\phi(t)=t$ for $0\leq t\leq s$, and $\phi(t)=r$ for $t\geq r$. Since $\phi(V(x))$ is constant in a neighborhood of the boundary $\partial \Omega$, it follows that it is a permissible test function in $(\ref{eq:mapp})$.  
The sign of $\phi''$ yields, 
\begin{align*}
0&=\int_{\Omega}[\Delta (\phi(V)) -b \cdot D (\phi(V)]\\
&=\int_{\Omega}[\phi''(V)|D V|^2+\phi'(V)\Delta V -\phi'(V)b \cdot D V]m\\
&\leq \int_{\Omega} \phi'(V)[\Delta V -b \cdot D V]m.
\end{align*}
The last integrand is zero outside $\{V\leq r\}$, equals $(\Delta V -b \cdot D V)m$ on $\{V\leq s\}$ and, is non-positive on $\{s\leq V\leq r\}$. Consequently,  
\begin{align*}
\int_{\{V\leq s\}}(\Delta V -b \cdot D V)m\geq 0.
\end{align*}
Letting $s\to \delta^{-C+1+\epsilon}$ yields 
\begin{align*}
\int_{\Omega_\delta}(\Delta V -b \cdot D V)m\geq 0.
\end{align*}
 Since $\Delta V -b \cdot D V\leq 0$ on $\Omega_\delta\backslash\Omega_{\delta_0}$, and  $\int_{\Omega_{\delta_0}}|\Delta V -b \cdot D V|m=S$, it follows that 
 \begin{align*}
 \int_{\Omega_\delta}|\Delta V -b \cdot D V|m\leq 2S.
 \end{align*}

\end{proof}


\bibliographystyle{abbrv}
\bibliography{biblio}

\begin{thebibliography}{10}

\bibitem{rus1}
V.~Bogachev and M.~R{\"o}ckner.
\newblock {A generalization of Khasminskii's theorem on the existence of
  invariant measures for locally integrable drifts. (Russian)}.
\newblock {\em Veroyatnost. i Primenen.}, 45(3):417--436, 2000.

\bibitem{notes}
P.~Cardaliaguet.
\newblock Notes on mean field games(from p.-l. lions’ lectures at coll\'ege
  de france), 2013.

\bibitem{gom}
D.~Gomes, E.~Pimentel, and V.~Voskanyan.
\newblock {\em Regularity Theory for Mean-Field Game Systems}.
\newblock SpringerBriefs in Mathematics. Springer International Publishing,
  2016.

\bibitem{c1}
M.~Huang, R.~P. Malham{\'e}, and P.~E. Caines.
\newblock Large population stochastic dynamic games: closed-loop mckean-vlasov
  systems and the nash certainty equivalence principle.
\newblock 2006.

\bibitem{ua}
W.~Huang, M.~Ji, Z.~Liu, and Y.~Yi.
\newblock Steady states of fokker–planck equations: I. existence.
\newblock {\em Journal of Dynamics and Differential Equations}, 27:721--742,
  2015.

\bibitem{lady}
O.~A. Ladyzhenskaya and N.~N. Uraltseva.
\newblock {\em Linear and quasilinear elliptic equations}.
\newblock Academic Press New York, 1968.

\bibitem{pll}
J.-M. Lasry and P.-L. Lions.
\newblock {Nonlinear elliptic equations with singular boundary conditions and
  stochastic control with state constraints}.
\newblock {\em Math. Ann.}, 283:583--630, 1989.

\bibitem{p1}
J.-M. Lasry and P.-L. Lions.
\newblock Mean field games.
\newblock {\em Japanese Journal of Mathematics}, 2:229--260, 2007.

\bibitem{france}
P.-L. Lions.
\newblock Cours au collège de france.
\newblock http://www.college-de-france.fr.

\bibitem{por2}
A.~Porretta and M.~Ricciardi.
\newblock Mean field games under invariance conditions for the state space.
\newblock {\em Communications in Partial Differential Equations}, 45:146 --
  190, 2019.

\bibitem{por}
A.~Porretta and L.~V{\'e}ron.
\newblock Asymptotic behaviour for the gradient of large solutions to some
  nonlinear elliptic equations.
\newblock {\em Adv. Nonlinear Stud.}, 2006.

\end{thebibliography}
\nocite{rus}
\nocite{gom}
\end{document}